\documentclass[fleqn,12pt]{article}
\usepackage{amsmath,amssymb,amsthm,xcolor,framed,esint,dsfont}
\usepackage[english]{babel}
\usepackage[margin=3cm]{geometry}
\usepackage{csquotes}
\usepackage[bookmarks]{hyperref}

\newcommand{\R}{{\mathbb R}}
\newcommand{\Z}{{\mathbb Z}}
\newcommand{\BV}{{\mathrm{BV}}}

\newcommand{\FT}{{\mathrm{FT}}}

\newcommand{\e}{\varepsilon}

\DeclareMathOperator*{\essinf}{ess\,inf}

\DeclareMathOperator{\dist}{dist}
\newcommand{\loc}{{\rm loc}}

\newtheorem{thm}{Theorem}
\newtheorem{prop}[thm]{Proposition}
\newtheorem{lem}[thm]{Lemma}

\theoremstyle{definition}

\newtheorem*{rem*}{Remark}

\title{On Lebesgue points of entropy solutions to the eikonal equation}
\date{\today}
\author{Xavier Lamy\thanks{Institut de Math\'ematiques de Toulouse; UMR 5219, Universit\'e de Toulouse; CNRS, UPS IMT, F-31062 Toulouse Cedex 9, France. Email: xlamy@math.univ-toulouse.fr.} \and Elio Marconi\thanks{Dipartimento di Matematica 'Tullio Levi Civita', Universit\`a di Padova, via Trieste 63, 35121 Padova (PD), Italy. Email: elio.marconi@unipd.it.}}
\begin{document}

\maketitle

\begin{abstract}
We consider entropy solutions to the eikonal equation $|\nabla u|=1$ in two space dimensions. These solutions are motivated by a class of variational problems and fail in general to have bounded variation. Nevertheless they share with BV functions, several of their fine properties: we show in particular that the set of non-Lebesgue points has co-dimension at least one.
\end{abstract}

\section{Introduction}
We consider an open set $\Omega\subset\R^2$ and $m\colon\Omega\to\R^2$ a solution of the eikonal equation
\begin{align}\label{e:eik}
|m|=1\text{ a.e., and }\nabla\cdot m=0\text{ in }\Omega.
\end{align}
We are interested in particular in solutions that arise as limits as $\e \to 0$ of vector fields $m_\e$ with equi-bounded energy $\sup_{\e>0} F_\e (m_\e,\Omega) < \infty$, where
\begin{align}\label{eq:AG}
&F_{\e}(m;\Omega)=\frac \e 2 \int_{\Omega}   |\nabla m|^2+ \frac 1{2\e}\int_\Omega (1-|m|^2)^2,\\
& m\colon\Omega\to \R^2,\qquad \nabla\cdot m =0,
\nonumber
\end{align}
are the functionals introduced by Aviles and Giga in \cite{AG86}.
 We refer to the introduction of \cite{JK00} for a description of several physical applications.

The notion of entropy, borrowed from the field of conservation laws, plays a fundamental role in the study of the singular limit as $\e\to 0$ of these functionals.
We say that a compactly supported function $\Phi \in C^\infty(\R^2,\R^2)$ is an \emph{entropy} for \eqref{e:eik} if for every open set $U$ and every smooth $m:U\to \R^2$ solving $\nabla\cdot m= 0$ and $|m|=1$ it holds $\nabla\cdot \Phi(m)=0$. 
It is shown in \cite{ADM,dkmo01} that functions with equi-bounded energy as $\e \to 0$ are pre-compact in $L^2(\Omega)$ and any limit is an entropy solution of \eqref{e:eik}: namely for every entropy $\Phi \in C^\infty(\R^2,\R^2)$ the distribution $\nabla\cdot \Phi(m)$ is a finite Radon measure.  
Remarkably, the same class of entropy solutions to \eqref{e:eik} contains the asymptotic domain of other families of functionals: see \cite{ARS, RS01} for two micromagnetics models.

It is shown in \cite{GL} that $m$ is an entropy solution if and only if the following kinetic equation  (introduced in \cite{jabinperthame01}) is satisfied:
\begin{align}\label{e:kin}
e^{is}\cdot\nabla_x \mathbf 1_{m(x)\cdot e^{is}>0} =\partial_s \sigma,\qquad \sigma\in\mathcal M_{\loc}(\Omega\times\R/2\pi\Z).
\end{align}
We denote by $\nu\in\mathcal M_{\loc}(\Omega)$ the entropy dissipation measure given by
\begin{align}\label{e:nu}
\nu(A)=|\sigma|(A\times\R/2\pi\Z),\qquad A\subset\Omega.
\end{align}
It is known \cite{ODL03} that $\mathcal H^1$-a.e. point $x\in\Omega$ at which $\nu(B_r(x))/r\to 0$ as $r\to 0^+$ is a vanishing mean oscillation (VMO) point of $m$, that is,
\begin{align*}
\fint_{B_r(x)}\left| m - \fint_{B_r(x)} m \right| \longrightarrow 0\quad\text{as }r\to 0^+.
\end{align*}
It is conjectured in \cite[\textit{Conjecture~1}(b')]{ODL03} that 
 $\mathcal H^1$-a.e. such point is in fact
 a Lebesgue point.
Our main result states that this conjecture is true under the additional assumption that $\nu(B_r(x))/r$ decays algebraically to 0. 

\begin{thm}\label{t:leb}
Let $m\colon\Omega\to\R^2$ be an entropy solution \eqref{e:kin} of the eikonal equation \eqref{e:eik}.
Then $\mathcal H^1$-a.e. $x\in\Omega$ such that $\lim_{r\to 0^+}\nu(B_r(x))/r^{1+a}=0$ 
for some $a>0$ is a Lebesgue point of $m$. 
In particular, the set of non Lebesgue points of $m$ has Hausdorff dimension at most 1.
\end{thm}

Analogs of Theorem~\ref{t:leb} have been obtained previously in \cite{LO18} for Burgers' equation, and in \cite{marconi22structure} for general scalar conservations laws.
To prove Theorem~\ref{t:leb}
we follow the scheme laid out in \cite{marconi22structure},
where it is shown that oscillations 
of averages $\fint_{B_r(x)}u$ of the solution $u$  are controlled by the
entropy dissipation.
This, together with the VMO property, implies the Lebesgue point property.
However, a key feature for the argument of \cite{marconi22structure} 
 is that the solution $u$ takes values in the ordered set $\R$.
Here our solution $m$ takes values in $\mathbb S^1$, 
and adapting the argument of \cite{marconi22structure} is not enough to conclude (see Proposition~\ref{p:weakdicho}).
Our proof of Theorem~\ref{t:leb}
 relies instead on the following dichotomy: 
 either the oscillations of $\fint_{B_r(x)}m$ are 
 controlled by the entropy dissipation $\nu$, 
or $m$ takes very different values in large subsets of $B_{R}(x)$ -- 
this second alternative is ruled out by the VMO property.
That dichotomy is made quantitative in the next statement.

\begin{prop}\label{p:dicho}
Assume $B_1\subset\Omega$. Let $r\in (0,1/2)$ and 
\begin{align}\label{e:h}
h
=h(r)
&=\max_{x_1,x_2\in \overline B_{2r}} \left| \fint_{B_r(x_1)}m -\fint_{B_r(x_2)}m\right|.
\end{align}
There exist absolute constants $c,\delta>0$ such that, if
\begin{align*}
R=\frac{32 r}{\delta h^2}\leq 1,
\end{align*}
then either
\begin{align}\label{e:dissip}
\nu(B_{R})\geq c\, h^{11} r,
\end{align}
or there exist $s_0\in\R$ such that
\begin{align}\label{e:nonVMO}
\left|B_{R}\cap \left\lbrace m\cdot e^{is} \geq -\frac 12 \right\rbrace \right| \geq c\, R^2
\quad\text{ for }\dist(s,\lbrace s_0,s_0+\pi\rbrace) \leq \pi/4.
\end{align}
\end{prop}

Here and in the rest of the article, we denote by $|A|$ the Lebesgue measure of  a measurable set $A\subset\R^d$. 
Theorem~\ref{t:leb} is a rather direct consequence of Proposition~\ref{p:dicho}, as we explain now.

\begin{proof}[Proof of Theorem~\ref{t:leb}]
Let $x\in\Omega$ be a VMO point of $m$ such that $\nu(B_r(x))/r^{1+a}\to 0$ for some $a>0$.
Translating and rescaling we assume without loss of generality that $x=0$ and $B_1\subset\Omega$. 
We claim that $h(r)= \mathcal O(r^b)$ for $b= a/(10+2a)>0$. 
This, together with 
the fact that $0$ is a VMO point of $m$, 
implies that $0$ is a Lebesgue point (see \cite[Lemma~4.6]{marconi22structure}).
To prove that $h(r)=\mathcal O(r^b)$ we argue by contradiction and assume that $h(r)/r^b\to\infty$ along a sequence $r\to 0^+$.
Then, along the same sequence,
\begin{align*}
&R =\frac{32r}{\delta h^2}=\frac {32} \delta r^{1-2b}\left(\frac{r^b}{h}\right)^2\to 0\quad\text{because }b < \frac 12,
\\
\text{ and }&\frac{R^{1+a}}{h^8r}=\frac{32^{1+a}}{\delta^{1+a}}\left(\frac{r^b}{h}\right)^{10+2a}\to 0.
\end{align*}
Therefore, applying Proposition~\ref{p:dicho} along the sequence $R\to 0$, the condition \eqref{e:dissip} cannot be satisfied because $\nu(B_R)/R^{1+a}\to 0$, so we have \eqref{e:nonVMO}. This contradicts the VMO property: for all small enough $R$, the projection $z_R\in\mathbb S^1$ of $\fint_{B_R}m$ onto $\mathbb S^1$ satisfies 
\begin{align}\label{e:zR}
\left|B_R \cap \lbrace |m-z_R|\geq \pi/12\rbrace \right| \leq \frac c 2 R^2.
\end{align}
But one can choose $s\in\R$ such that $\dist(s,\lbrace s_0,s_0+\pi\rbrace)\leq\pi/4$ and
\begin{align*}
z \cdot e^{is}\geq -\frac 12\quad \Longrightarrow \quad |z-z_R|\geq \pi/12,
\end{align*}
for any $z\in \mathbb S^1$ (if $z_R=e^{is_R}$, any $s\in [s_R + 3\pi/4,s_R +5\pi/4]$ has that property). 
According to \eqref{e:nonVMO} this implies
$|B_R \cap \lbrace |m-z_R|\geq \pi/12\rbrace | \geq cR^2$, 
in contradiction with \eqref{e:zR}. 
Hence we have proved that $x$ is a Lebesgue point. 
The estimate on the Hausdorff dimension of non Lebesgue points follows via a covering argument (see e.g. \cite[Theorem~2.56]{ambrosio}).
\end{proof}

The proof of Proposition~\ref{p:dicho} has two main ingredients.
The first ingredient consists in adapting the arguments of \cite{marconi22structure} to prove a
dichotomy similar to Proposition~\ref{p:dicho}, 
but where the second option \eqref{e:nonVMO} is replaced by a statement which is not strong enough to conclude.

\begin{prop}\label{p:weakdicho}
Let $r\in (0,1/2)$ and $h$ be as in Proposition~\ref{p:dicho}.
There exist absolute constants $c,\delta>0$ such that, if
$R=32r/(\delta h^2)\leq 1$,
then we have either $\nu(B_{R})\geq c h^{11} r$, or
\begin{align}\label{e:nonVMOweak}
\left|B_{R/2}\cap \left\lbrace m\cdot m_0 \geq  \frac 12 \right\rbrace \right| \geq c hr^2 
\quad
\text{and }
\quad
\left|B_{R/2}\cap \left\lbrace m\cdot m_0 \leq -\frac 12 \right\rbrace \right| \geq c hr^2,
\end{align}
for some $m_0\in\mathbb S^1$.
\end{prop}

The main idea behind the argument in \cite{marconi22structure} is that a large value of $h$ implies the existence of a configuration which would be impossible in the absence of entropy dissipation. In the presence of dissipation, such configuration  provides a lower bound on the dissipation, and there is no dichotomy.
Here instead, not all configurations created by large values of $h$ can be ruled out in the absence of dissipation:
 in particular the vortex solution $m(x)= x^\perp / |x|$ has zero dissipation but the values of $h(r)$ around the origin are not vanishing.
This is reflected in the second alternative \eqref{e:nonVMOweak} of the dichotomy.

The second ingredient in our proof of Proposition~\ref{p:dicho} consists in using the methods developed in \cite{marconi21micromag,marconi21ellipse,CHLM22,stab} in order to pass from \eqref{e:nonVMOweak} to \eqref{e:nonVMO}.

\begin{prop}\label{p:nonVMOimprov}
Let $m_0=e^{is_0}\in\mathbb S^1$, and $R>0$ such that $B_R\subset\Omega$. Then we have either
\begin{align}\label{e:dissipXpm}
\nu(B_{R})\geq \frac{c}{R} \min (|X_+|,|X_-|),\qquad
X_\pm
= B_{R/2}\cap \lbrace \pm m\cdot m_0\geq 1/2 \rbrace,
\end{align}
or $\nu(B_R)\geq c R$, or \eqref{e:nonVMO}, for some absolute constant $c>0$.
\end{prop}

Proposition~\ref{p:dicho} follows readily from Proposition~\ref{p:weakdicho} and Proposition~\ref{p:nonVMOimprov}. 
Thanks to Proposition~\ref{p:weakdicho}, we know indeed that either $\nu(B_R)\geq c h^{11}r$, in which case we are done, or estimate \eqref{e:nonVMOweak} is valid.
But according to Proposition~\ref{p:nonVMOimprov}, if \eqref{e:nonVMOweak} is satisfied, then we have either $\nu(B_R)\geq c hr^2/R \geq c h^{11} r$, or $\nu(B_R)\geq c R\geq c h^{11}r$, or \eqref{e:nonVMO}. 
In all cases, Proposition~\ref{p:dicho} is verified.

The proofs of Proposition~\ref{p:weakdicho} and Proposition~\ref{p:nonVMOimprov} are presented in 
 Section~\ref{s:weakdicho} and Section~\ref{s:nonVMOimprov}.
 
 \subsection*{Notations.} We denote by $|A|$ the Lebesgue measure of a set $A\subset\R^d$. We use the symbol $\gtrsim$ to signify inequality up to an absolute mutliplicative constant.
 
 \subsection*{Acknowledgements.} 
 Part of this work was completed during X.L.'s stay at the EPFL's Institute of Mathematics, which he thanks for their hospitality.  X.L. received  support  from ANR project ANR-22-CE40-0006-01.
E.M. acknowledges the support received from the SNF Grant 182565 and the European Union's Horizon 2020 research and innovation program under the Marie Sk\l odowska-Curie grant No. 101025032.

\section{Proof of Proposition~\ref{p:weakdicho}}\label{s:weakdicho}

Let $x_1,x_2$ attain the maximum in the definition \eqref{e:h} of $h$, and define, for $j=1,2$, $\rho_j(s)$ as the proportion of points  $x\in B_r(x_j)$ at which $m(x)$ lies in the semi-circle of direction $e^{is}$, that is, for every $s\in\R/2\pi\Z$, we set
\begin{align*}
\rho_j(s)&=\frac{1}{|B_r|}
\left|
B_r(x_j)\cap
\left\lbrace
m\cdot e^{is}>0
\right\rbrace\right| = \frac1{|B_r|}\int_{B_r(x_j)}\mathbf{1}_{E_m} (x,s) dx,
\end{align*}
where
\begin{align}\label{e:def_Em}
E_m=\left\lbrace (x,s)\in\Omega\times\R/2\pi\Z\colon m(x)\cdot e^{is}>0\right\rbrace.
\end{align}

Note that $|\rho_j|\leq 1$ and, since for every $x \in \Omega$ it holds $|D_s\mathbf{1}_{E_m} (x,\cdot)| ( \R/2\pi \Z) =2$,  then  $\rho_j\in BV(\R/2\pi \Z)$ with $|D\rho_j|(\R/2\pi\Z)\leq 2$.
Moreover, by Fubini theorem, these functions satisfy the identities
 \begin{align*}
 \int_{\mathbb R/2\pi\Z}e^{is}\rho_j(s)\, ds = \int_{\mathbb R/2\pi\Z} \fint_{B_r(x_j)}\mathbf{1}_{E_m} (x,s) e^{is} dx ds =
2 \fint_{B_r(x_j)}m(x)\, dx.
 \end{align*}
 For $s\in\R$ and $\rho>0$ we denote by $I_\rho(s)$ the segment
 \begin{align*}
 I_\rho(s)=[s-\rho,s+\rho].
 \end{align*} 
For a small enough absolute constant $\delta\in (0,1)$, the subset $S\subset \R/2\pi\Z$ given by
\begin{align*}
S=
\left\lbrace
s\in\R/2\pi\Z\colon
(|D\rho_1|+|D\rho_2|)(I_{\delta h^2}(s))\geq \frac h {4\pi}
\right\rbrace,
\end{align*}
satisfies $|S|\leq  h /2$ (as
follows e.g. from a Besicovitch covering argument).
Thus  we have
\begin{align*}
h &=\frac 12\left|
\int_{\mathbb R/2\pi\Z}e^{is}\rho_1(s)\, ds -\int_{\R/2\pi\Z}e^{is}\rho_2(s)\, ds
\right|
\leq 
\frac 12
\int_{\R/2\pi\Z}|\rho_1(s)-\rho_2(s)|\, ds \\
&
\leq \frac 12 \int_{(\R/2\pi\Z)\setminus S} |\rho_1(s)-\rho_2(s)|\, ds + \frac{h}{2}.
\end{align*}
We may therefore find $s\in \R/2\pi\Z$ such that $s\notin S$ and $|\rho_1(s)-\rho_2(s)|\geq h/{2\pi}$. We assume without loss of generality that $\rho_1(s)-\rho_2(s)\geq h/2\pi$, and by definition of $S$ we deduce
\begin{align*}
\inf_{I_{\delta h^2}(s) }\rho_1- \sup_{I_{\delta h^2}(s)}\rho_2\geq \frac h {4\pi} .
\end{align*}
In particular, setting $s_0=s-\pi/2 -3\delta h^2/4$, we have
\begin{align*}
&\inf_{I_{\delta h^2/4}(s_0+\pi/2) }\rho_1- \sup_{I_{\delta h^2/4}(s_0 +\pi/2 +\delta h^2)}\rho_2 \geq \frac h {4\pi} ,
\\
&\inf_{I_{\delta h^2/4}(s_0+\pi/2 +\delta h^2) }\rho_1- \sup_{I_{\delta h^2/4}(s_0+\pi/2)}\rho_2 \geq \frac h {4\pi}.
\end{align*}
As $\rho_j(s+\pi)=1-\rho_j(s)$ for a.e. $s\in\R/2\pi\Z$, this implies
\begin{align}
&\essinf_{I_{\delta h^2/4}(s_0+\pi/2) }\rho_1 + \essinf_{I_{\delta h^2/4}(s_0-\pi/2 +\delta h^2)}\rho_2 \geq 1 +  \frac h {4\pi} ,
\label{e:rho1rho2+}
\\
&\essinf_{I_{\delta h^2/4}(s_0 +\pi/2 +\delta h^2) }\rho_1 + \essinf_{I_{\delta h^2/4}(s_0-\pi/2)}\rho_2 \geq 1+ \frac h {4\pi}.
\label{e:rho1rho2-}
\end{align}
The relevance of \eqref{e:rho1rho2+}-\eqref{e:rho1rho2-} comes from the following geometric observation.
Given two directions $s_1\in I_{\delta h^2/4}(s_0+\pi/2)$ and $s_2\in I_{\delta h^2/4}(s_0-\pi/2+\delta h^2)$ 
and two points $y_1\in B_r(x_1)\cap \lbrace m\cdot e^{is_1}>0\rbrace$, $y_2\in B_r(x_2)\cap\lbrace m\cdot e^{is_2}>0\rbrace$, 
we have $|s_1-s_2|\geq \delta h^2$, and  the two lines $y_j +\R e^{is_j}$ intersect at a point  $z\in B_{8r/(\delta h^2)}$.
In the absence of dissipation, one would have $m(z)\cdot e^{is_j}>0$ for $j=1,2$, and therefore $m(z)\cdot e^{is_0}\geq \cos(2\delta h^2)\geq 1/2$.
The last lower bound is valid provided $\delta\leq \pi/24$,  since $|h|\leq 2$.
The same argument with $s_1\in I_{\delta h^2/4}(s_0+\pi/2+\delta h^2)$ 
and $s_2\in I_{\delta h^2/4}(s_0-\pi/2)$ 
 implies instead $m(z)\cdot e^{is_0}\leq -1/2$.

\medskip

Thanks to the techniques in \cite{marconi22structure}, in the presence of dissipation this can be made quantitative. 
The main idea is that \eqref{e:kin} provides an estimate on the difference between the `epigraph' $E_m$ defined in \eqref{e:def_Em}
and its free transport $\FT(E_m,t)$, where the free transport operator $\FT(\cdot,t)$ is defined for $t\in\R$ by
\begin{align*}
\FT(E,t)=\left\lbrace (x,s)\in\Omega\times\R/2\pi\Z
\colon (x-t e^{is},s)\in E\right\rbrace.
\end{align*}
\begin{lem}\label{l:FT}
Let $t\in \R$ and $\rho>0$ such that $B_{\rho + |t|}\subset\Omega$.
For all  $\phi\in C_c^1(B_\rho\times \R/2\pi\Z)$ we have
\begin{align*}
\int_{\Omega\times\R/2\pi\Z}\phi(x,s)
\left( \mathbf 1_{\FT(E_m,t)}-\mathbf 1_{E_m}\right)
\, dxds
\leq 
\left( |t|\, \|\partial_s\phi\|_\infty +t^2\|\nabla_x\phi\|_\infty\right)
\nu(B_{\rho+|t|}).
\end{align*}
\end{lem}
\begin{proof}[Proof of Lemma~\ref{l:FT}]
Define $\chi,\chi^{\FT}\colon [-|t|,|t|]\times\Omega\times\R/2\pi\Z\to\R$ by
\begin{align*}
\chi(\tau,x,s)=\mathbf 1_{(x,s)\in E_m},\qquad \chi^{\FT}(\tau,x,s)=\mathbf 1_{(x,s)\in \FT(E_m,\tau)}
=\chi(x-\tau e^{is},s).
\end{align*}
so we have, in the sense of distributions,
\begin{align*}
\partial_\tau\chi +e^{is}\cdot\nabla_x \chi &=\partial_s\sigma(x,s),
\qquad
\partial_\tau\chi^{\FT} +e^{is}\cdot\nabla_x \chi^{\FT} =0.
\end{align*}
Setting $\hat\chi=\chi^{\FT}-\chi$, and $\psi(\tau,x,s)=\phi(x+e^{is}(t-\tau),s)$ which satisfies $\partial_\tau\psi +e^{is}\cdot\nabla_x\psi=0$, we deduce
\begin{align*}
\partial_\tau \left[\psi\hat\chi \right] +e^{is}\cdot\nabla_x \left[\psi\hat\chi \right] =-\psi \partial_s\sigma.
\end{align*}
Integrating with respect to $(x,s)$ (this is formal but makes sense distributionnally) we deduce
\begin{align*}
\frac{d}{d\tau} \int_{\Omega\times\R/2\pi\Z} \psi\hat\chi\, dx ds 
=\int_{\Omega\times\R/2\pi\Z} \partial_s\psi\, d\sigma(x,s).
\end{align*}
Integrating this from $0$ to $t$  and recalling $\nu(A)=|\sigma|(A\times\R/2\pi\Z)$ for $A\subset\Omega$, we obtain
\begin{align*}
\int_{\Omega\times\R/2\pi\Z}\phi(x,s)
\left( \mathbf 1_{\FT(E_m,t)}-\mathbf 1_{E_m}\right)
\, dxds
&=
\int_0^t \int_{\Omega\times \R/2\pi\Z} \partial_s\psi d\sigma\, d\tau \\
& \leq |t|\, \|\partial_s\psi\|_\infty |\nu|(B_{\rho+|t|}\times\R/2\pi\Z).
\end{align*}
Noting that $\|\partial_s\psi\|_\infty\leq \|\partial_s\phi\|_\infty +|t| \,\|\nabla_x\phi\|_\infty$ completes the proof.
\end{proof}

Equipped with Lemma~\ref{l:FT} we continue the proof of Proposition~\ref{p:weakdicho}. 
First we make use of \eqref{e:rho1rho2+}.
We define $\hat z\in\R^2$ as the intersection of the lines $x_1+\R e^{i(s_0+\pi/2)}$ and $x_2+\R e^{i(s_0-\pi/2+\delta h^2)}$, that is,
\begin{align*}
x_1+t_1 e^{i(s_0 +\pi/2)}=x_2+t_2 e^{i(s_0-\pi/2+\delta h^2)}=\hat z,
\end{align*}
for some $t_1,t_2\in\R$.
Since $|x_1-x_2|\leq 4r$, we have
\begin{align}\label{e:est_t_i}
|t_1|,|t_2|\leq \frac{4r}{\sin(\delta h^2)}\leq \frac{8r}{\delta h^2},
\end{align}
and therefore  $B_r(\hat z)\subset B_{R/2}$.
We will use Lemma~\ref{l:FT} to compare $E_m$ with $\FT(E_m,t_1)$ and $\FT(E_m,t_2)$ on $B_r(\hat z)$. 
We define
\begin{align*}
C_1 
&= B_r(\hat z)\times I_{c}(s_0+\pi/2), \qquad C_2 = B_r(\hat z)\times I_{c}(s_0-\pi/2+\delta h^2), \\
A_1
&
=E_m\cap C_1, \qquad A_2 = E_m \cap C_2
\end{align*}
with $c= \delta h^3/128\pi \le \delta h^2/4$.
and their free transport counterparts
\begin{align*}
A_1^{\FT}
=\FT(E_m,t_1)\cap C_1, \qquad
A_2^{\FT}
=\FT(E_m,t_2)\cap C_2.
\end{align*}
We estimate 
\begin{align*}
|A_1^{\FT}| = &~  | E_m \cap \FT(\cdot, t_1)^{-1}(C_1)| \\
\ge & ~ |E_m \cap (B_r(x_1)\times I_c(s_0 + \pi/2))| - | \FT(\cdot,t_1)^{-1}(C_1) \setminus  (B_r(x_1)\times I_c(s_0 +\pi/2))| \\
= &~ \int_{s_0 +\frac\pi2 -c}^{s_0 +\frac\pi2 +c}\rho_1(s) ds - | \FT(\cdot,t_1)^{-1}(C_1) \setminus  (B_r(x_1)\times I_c(s_0 +\pi/2))|.
\end{align*}
Moreover
\begin{align*}
| \FT(\cdot,t_1)^{-1}(C_1) \setminus  (B_r(x_1)\times I_c(s_0 +\pi/2))| 
&= 
 \int_{s_0 +\frac\pi2 -c}^{s_0 +\frac\pi2 +c}|B_r(\hat z - t_1 e^{is}) \setminus B_r(x_1) | ds \\
&
\le  2r \int_{s_0 +\frac\pi2 -c}^{s_0 +\frac\pi2 +c} | \hat z - t_1 e^{is} - x_1| ds \\
&
\le 
2r \int_{s_0 +\frac\pi2 -c}^{s_0 +\frac\pi2 +c} |t_1| | e^{i (s_0 + \pi/2)}-e^{is}| ds \\
&
\le 
32 \frac{c^2 r^2}{\delta h^2},
\end{align*}
where in the last inequality we used \eqref{e:est_t_i}.
Therefore we have
\begin{align}\label{e:A1FT}
|A_1^{\FT}| \ge  \int_{s_0 +\frac\pi 2 -c}^{s_0 +\frac\pi2 +c}\rho_1(s) ds -32 \frac{c^2 r^2}{\delta h^2},
\end{align}
and similarly
\begin{align}\label{e:A2FT}
|A_2^{\FT}| \ge  
\int_{  s_0 -\frac\pi 2 +\delta h^2 -c}^{ s_0 - \frac\pi2 + \delta h^2 +c}
\rho_2(s) ds -32 \frac{c^2 r^2}{\delta h^2}.
\end{align}

From \eqref{e:rho1rho2+} we know that
\begin{align*}
\rho_1\left(s_0+ \frac\pi 2 +s\right) + \rho_2\left(s_0-\frac\pi 2 +\delta h^2 +s\right)\geq 1+\frac h{4\pi}\qquad\text{for all }|s|\leq \frac{\delta}{4}h^2.
\end{align*}
Integrating this inequality in $s\in[-c,c]$, it follows from \eqref{e:A1FT},\eqref{e:A2FT} that
\begin{align}\label{e:A1A2FT}
|A_1^{\FT}| +|A_2^{\FT}|\geq 2c|B_r|\left(1+\frac h{4\pi}\right) -64 \frac{c^2 r^2}{\delta h^2}\ge  2c|B_r|\left(1+\frac h{8\pi}\right),
\end{align}
by the choice $c=\delta h^3/128\pi$.
Next we consider two cases, depending on whether $A_1$ and $A_2$ satisfy a similar inequality. 

\medskip

\noindent\textbf{Case 1.} Assume first that
\begin{align*}
|A_1|+|A_2|\geq 2c|B_r|\left(1+\frac h{16\pi}\right),
\end{align*}
then
\begin{align*}
|\pi_x(A_1)|+ |\pi_x(A_2)|\geq  |B_r|\left(1+ \frac{h}{16\pi}\right).
\end{align*}
Moreover, since $\pi_x(A_1) \cup \pi_x(A_2) \subset B_r(\hat z)$, it follows that $A:= \pi_x(A_1) \cap \pi_x(A_2)$ satisfies $|A|\ge h|B_r|/16$.
By construction, we have
\begin{align*}
A &=\Big\lbrace x\in B_r(\hat z)\colon  \exists s_1\in I_{c}(s_0+\pi/2),\, s_2\in I_{c}(s_0-\frac\pi 2 +\delta h^2),\\
&\hspace{7em} m(x)\cdot e^{is_1}>0\text{ and }m(x)\cdot e^{is_2}>0 \Big\rbrace \\
& \subset B_r(\hat z) \cap \lbrace m\cdot e^{is_0} \geq \cos(2\delta h^2)\rbrace,
\end{align*}
so this implies
\begin{align}\label{e:case1+}
\left| B_{R/2}\cap \left\lbrace m\cdot m_0\geq \frac 12\right\rbrace
\right|\gtrsim hr^2.
\end{align}

\medskip

\noindent\textbf{Case 2.} Assume now that 
\begin{align*}
|A_1|+|A_2|< 2c |B_r|\left(1+\frac h{16\pi}\right).
\end{align*}
Then using \eqref{e:A1A2FT} we obtain
\begin{align*}
|A_1^{\FT}|-|A_1| + |A_2^{\FT}|-|A_2| 
>
2c |B_r| \frac h{16\pi},
\end{align*}
so either $|A_1^{\FT}|-|A_1|$ or $|A_2^{\FT}|-|A_2|$ is larger than half the right-hand side. We consider without loss of generality only the first case:
\begin{align*}
|A_1^{\FT}|-|A_1|  
>
|B_r| \frac{ch}{16\pi}.
\end{align*}
This implies a lower bound on the entropy dissipation $\nu(B_{R})$ thanks to Lemma~\ref{l:FT}. 
Specifically, we apply Lemma~\ref{l:FT} to $t=t_1$ and $\phi\in C_c^\infty(B_{2r(\hat z)}\times I_{2c}(s_0+\pi/2))$ such that
\begin{align*}
\mathbf 1_{x\in B_r(\hat z)}\mathbf 1_{s\in I_{c}(s_0+\pi/2)}
\leq
\phi(x,s)
\leq 
\mathbf 1_{x\in B_{(1+\e)r}(\hat z)}\mathbf 1_{s\in I_{(1+\e)c}(s_0+\pi/2)},
\end{align*}
and $|\partial_s\phi|\leq 2/(\e c)$, $|\nabla_x\phi|\leq 2/(\e r)$.
We choose $\e=h/192\pi$ to ensure
\begin{align*}
\left|
\left(B_{(1+\e)r}(\hat z)\times I_{(1+\e)c}\right)
\setminus
\left(B_{r}(\hat z)\times I_{c}\right)
\right| \leq \frac{ch}{32\pi}|B_r|.
\end{align*}
Since $|t_1|\leq 8r/(\delta h^2)$ and $B_{2r+|t_1|}(\hat z)\subset B_R$, we deduce $\nu(B_R)\gtrsim \delta^3 h^{11} r \gtrsim h^{11}r$.

\medskip

Similarly, using \eqref{e:rho1rho2-} we have two cases: either 
\begin{align}\label{e:case1-}
\left| B_{R/2}\cap \left\lbrace m\cdot m_0\leq -\frac 12\right\rbrace
\right|\gtrsim hr^2,
\end{align}
or $\nu(B_R)\gtrsim  h^{11} r$. 
So gathering all cases, we see that either both \eqref{e:case1+} and \eqref{e:case1-} are satisfied, or $\nu(B_R)\gtrsim h^{11}r$, 
which is exactly the dichotomy of Proposition~\ref{p:weakdicho}.
\qed

\section{Proof of Proposition~\ref{p:nonVMOimprov}}\label{s:nonVMOimprov}

In order to prove Proposition~\ref{p:nonVMOimprov}, we briefly recall from \cite{marconi21ellipse} the notion of \emph{Lagrangian representation} of an entropy solution $m$ of the eikonal equation.
In \cite{marconi21ellipse,marconi21micromag} the second author shows the existence of a finite non-negative Radon measure $\omega$ on the set of curves 
\begin{align*}
\Gamma= \Big\{ (\gamma,t^-_\gamma,t^+_\gamma)\colon
& 0\le t^-_\gamma\le t^+_\gamma\le 1, \\
&
\gamma=(\gamma_x,\gamma_s)\in \BV((t^-_\gamma,t^+_\gamma);\Omega \times \R/2\pi \Z),
\\
&
 \gamma_x \text{ is Lipschitz} \Big\},
\end{align*} 
 with the following three properties:
\begin{itemize}
\item for every $t\in (0,1)$, the pushforward of $\omega$, restricted to the section
$\Gamma(t)=\lbrace (\gamma,t^-_\gamma,t^+_\gamma)\in\Gamma\colon t_\gamma^- < t <t_\gamma^+\rbrace$, by the evaluation map $e_t\colon \gamma\mapsto \gamma(t)$ (a right-continuous representative of $\gamma_s$ is always considered), is uniform on the `epigraph' $E_m=\lbrace m(x)\cdot e^{is}>0\rbrace$, that is,
\begin{equation}\label{e:lag}
( e_t)_\sharp \left[ \omega \lfloor  \Gamma(t)\right]= \mathbf 1_{m(x)\cdot e^{is}>0}\, dx\, ds ;
\end{equation}
\item the measure $\omega$ is concentrated on curves $(\gamma,t^-_\gamma,t^+_\gamma)\in  \Gamma$ solving the characteristic equation,
\begin{equation}\label{e:charac}
\dot\gamma_x(t)= e^{i \gamma_s(t)}\qquad\text{for a.e. }t\in (t^-_\gamma,t^+_\gamma);
\end{equation}
\item the entropy dissipation measure \eqref{e:nu} disintegrates along the Lagrangian curves as
\begin{equation}\label{e:disint}
 \nu(A)=
 \int_\Gamma \mu_\gamma (\gamma_x^{-1}(A))\, d\omega(\gamma)\qquad\text{for all measurable }A\subset\Omega,
\end{equation}
 where $\mu_\gamma =|D_t\gamma_s|$, with the convention that a jump of $\gamma_s$ from $s^-$ to $s^+$ at time $t_0\in (t_\gamma^-,t_\gamma^+)$ contributes $\dist_{\R/2\pi\Z}(s^-,s^+)\delta_{t=t_0}$ to the jump part of $\mu_\gamma$ (see \cite[Proposition~2.5]{marconi21ellipse}).
\end{itemize}
Moreover, the Lagrangian property \eqref{e:lag} implies that $\omega$ is concentrated on curves $\gamma$ such that $\gamma_x(t)$ is a Lebesgue point of $m$ with $m(\gamma_x(t))\cdot e^{i\gamma_s(t^+)}>0$, for a.e. $t\in (0,1)$  \cite[Lemma~2.7]{marconi21ellipse}.
We denote by $\Gamma_g\subset\Gamma$ the full-measure subset of Lagrangian curves which satisfy that property together with the characteristic equation \eqref{e:charac}.

The proof of Proposition~\ref{p:nonVMOimprov} is based on two main tools.
The first, Lemma~\ref{l:dichocurve}, is a dichotomy stating that either Lagrangian curves passing through a given set create a lot of dissipation, or one can find an almost-straight Lagrangian curve passing through that set. 
The second (\cite[Lemma~5.2]{stab}, a slightly more precise version of \cite[Lemma~3.1]{marconi21ellipse}, itself adapted from \cite[Lemma~22]{marconi21micromag}) is another dichotomy: given an almost-straight Lagrangian curve, either the density of points at which $m$ lies in the semi-circle indicated by the $s$-component of that curve is high, or a lot of dissipation must be created.
The succession of these two dichotomies is reflected in the three alternatives in the conclusion of Proposition~\ref{p:nonVMOimprov}.
We first state and prove the first tool, and then proceed to the proof of Proposition~\ref{p:nonVMOimprov}.

\begin{lem}\label{l:dichocurve}
For any $R>0$ such that $B_R\subset\Omega$, any measurable set $A\subset B_{R}\times\R/2\pi\Z$, and any $\eta\in (0,1)$, we have either
\begin{align}\label{e:Adissip}
\nu(B_R)\gtrsim \frac{\eta}{ R}\left|\lbrace (x,s)\in A\colon m(x)\cdot e^{is}>0\rbrace\right|,
\end{align}
or there exists a curve $\gamma\in\Gamma_g$ and a connected component $J$ of $\gamma_x^{-1}(B_R)$ such that
\begin{align*}
J\cap\gamma^{-1}(A)\neq \emptyset
\quad\text{and}\quad
\mu_\gamma(J)\leq \eta.
\end{align*}
\end{lem}

\begin{proof}[Proof of Lemma~\ref{l:dichocurve}]
Assume that the second alternative of Lemma~\ref{l:dichocurve} is not verified: for every curve $\gamma\in \Gamma_g$ and every connected component $J$ of $\gamma_x^{-1}(B_R)$ intersecting $\gamma^{-1}(A)$, we have $\mu_\gamma(J) >\eta$. 
Then we claim that
\begin{align}\label{e:muT}
\mu_\gamma(\gamma_x^{-1}(B_R)) &\gtrsim \frac{\eta }{ R}T(\gamma),
\qquad
T(\gamma) =\left|
\lbrace t\in (t_\gamma^-,t_\gamma^+)\colon \gamma(t)\in A \rbrace
\right|,
\end{align}
 for all $\gamma\in \Gamma_g$.
 To prove \eqref{e:muT}, denote by $J_k=(t_k^-,t_k^+)$ the connected components of $\gamma_x^{-1}(B_R)$ which intersect $\gamma^{-1}(A)$. 
We show next that $\mu_\gamma(J_k)\gtrsim \eta |J_k|/ R$ for all $k$.
On the one hand, if $|J_k|\leq 4R$ then $\mu_\gamma(J_k)\gtrsim \eta |J_k|/ R$ because $\mu_\gamma(J_k)> \eta$ by assumption.
On the other hand, from the characteristic equation \eqref{e:charac}
and the definition of $\mu_\gamma=|D_t\gamma_s|$, we have the inequality
\begin{align*}
|\gamma_x(t_2)-\gamma_x(t_1)-e^{i\gamma_s(t_1)}(t_2-t_1)|\leq \mu_\gamma([t_1,t_2])|t_2-t_1|,
\end{align*}
and we deduce that in any
 any interval $J\subset (0,1)$ such that $\gamma_x(J)\subset B_R$ and $|J|\geq 4R$, we must have $\mu_\gamma(J) \geq 1/2$.
Therefore, if $|J_k|\geq 4R$, cutting $J_k$ in disjoint subintervals of length between $4R$ and $8R$, we obtain that $\mu_\gamma(J_k)  \gtrsim |J_k|/R \geq \eta |J_k|/ R$. 
So we have 
\begin{align*}
\mu_\gamma(\gamma_x^{-1}(B_R)) \geq \sum_k \mu_\gamma(J_k) \gtrsim \frac{\eta}{R}\sum_k |J_k|,
\end{align*}
which implies \eqref{e:muT} since $\gamma^{-1}(A)\subset \bigcup_k J_k$.
From \eqref{e:muT} and the fact that $\omega(\Gamma\setminus \Gamma_g)=0$ we infer
\begin{align*}
\nu(B_{R})&= \int_{\Gamma} \mu_\gamma(\gamma_x^{-1}(B_R ))\, d\omega(\gamma) 
\gtrsim \frac{\eta}{R}\int_{\Gamma} T(\gamma)\, d\omega(\gamma),
\end{align*}
where the first equality comes from the disintegration \eqref{e:disint}.
Making use of the Lagrangian property \eqref{e:lag} 
to rewrite the last expression, 
we see that it is precisely equal to the right-hand side of \eqref{e:Adissip}, which concludes the proof of Lemma~\ref{l:dichocurve}.
\end{proof}

\begin{proof}[Proof of Proposition~\ref{p:nonVMOimprov}]
We  recall that $m_0=e^{is_0}$ and the sets $X_\pm$ are defined by
\begin{align*}
X_\pm =B_{R/2}\cap \lbrace \pm m\cdot e^{is_0}\geq 1/2\rbrace.
\end{align*}
For any $\hat s\in [s_0-\pi/4,s_0+\pi/4]$, we apply Lemma~\ref{l:dichocurve} to $A(\hat s)=B_{R/2}\times I_\eta(\hat s)$, where $I_\eta(\hat s)=[\hat s-\eta,\hat s+\eta]$. If $\eta\in (0,\pi/12)$ then we have $m(x)\cdot e^{is}>0$ for all $(x,s)\in X_+\times I_\eta(\hat s)$, and therefore
\begin{align*}
\left|\lbrace (x,s)\in A(\hat s)\colon m(x)\cdot e^{is}>0\rbrace\right|
\geq \eta |X_+|.
\end{align*}
So we have either $\nu(B_R)\gtrsim\eta^2|X_+|/R$, or there exists a curve $\gamma\in\Gamma_g$ and  a connected component $J$ of $\gamma_x^{-1}(B_R)$ intersecting $A(\hat s)$ such that $\mu_\gamma(J)<\eta$. In that second case, applying \cite[Lemma~5.2]{stab} we deduce that either $\nu(B_R)\gtrsim \eta^3 R$ or $|B_R\cap \lbrace m\cdot e^{i\hat s}\geq -2\eta\rbrace|\gtrsim \eta R^2$. 
 We fix $\eta =1/4$ and summarize the preceding discussion: for all $\hat s\in [s_0-\pi/4,s_0+\pi/4]$, we have 
\begin{align*}
\nu(B_R)\gtrsim \frac{|X_+|}{R},\quad\text{ or }\nu(B_R)\gtrsim R,\quad\text{ or }
\left|B_R\cap \lbrace m\cdot e^{i\hat s}\geq -1/2\rbrace\right|\gtrsim R^2.
\end{align*}
Similarly, for all $\hat s\in [s_0+3\pi/4,s_0+5\pi/4]$, we have
\begin{align*}
\nu(B_R)\gtrsim \frac{|X_-|}{R},\quad\text{ or }\nu(B_R)\gtrsim R,\quad\text{ or }
\left|B_R\cap \lbrace m\cdot e^{i\hat s}\geq -1/2\rbrace\right|\gtrsim R^2.
\end{align*}
We conclude that we have either \eqref{e:dissipXpm}, or $\nu(B_R)\gtrsim R$, or
\begin{align*}
\left|B_R\cap \lbrace m\cdot e^{i s}\geq -1/2\rbrace\right|\gtrsim R^2,
\end{align*}
for all $ s\in [s_0-\pi/4,s_0+\pi/4]\cup [s_0+3\pi/4,s_0+5\pi/4]$. This corresponds exactly to the three alternatives in the statement of Proposition~\ref{p:nonVMOimprov}.
\end{proof}

\bibliographystyle{acm}
\bibliography{aviles_giga}

\begin{thebibliography}{10}

\bibitem{ARS}
{\sc Alouges, F., Rivi\`ere, T., and Serfaty, S.}
\newblock N\'eel and cross-tie wall energies for planar micromagnetic
  configurations.
\newblock {\em ESAIM Control Optim. Calc. Var. 8\/} (2002), 31--68.
\newblock A tribute to J. L. Lions.

\bibitem{ADM}
{\sc Ambrosio, L., De~Lellis, C., and Mantegazza, C.}
\newblock Line energies for gradient vector fields in the plane.
\newblock {\em Calc. Var. Partial Differential Equations 9}, 4 (1999),
  327--255.

\bibitem{ambrosio}
{\sc Ambrosio, L., Fusco, N., and Pallara, D.}
\newblock {\em Functions of bounded variation and free discontinuity problems}.
\newblock Oxford Mathematical Monographs. The Clarendon Press, Oxford
  University Press, New York, 2000.

\bibitem{AG86}
{\sc Aviles, P., and Giga, Y.}
\newblock A mathematical problem related to the physical theory of liquid
  crystal configurations.
\newblock In {\em Miniconference on geometry and partial differential
  equations, 2 ({C}anberra, 1986)}, vol.~12 of {\em Proc. Centre Math. Anal.
  Austral. Nat. Univ.} Austral. Nat. Univ., Canberra, 1987, pp.~1--16.

\bibitem{CHLM22}
{\sc Contreras~Hip, A.~A., Lamy, X., and Marconi, E.}
\newblock Generalized characteristics for finite entropy solutions of
  {B}urgers' equation.
\newblock {\em Nonlinear Anal. 219\/} (2022), Paper No. 112804.

\bibitem{ODL03}
{\sc De~Lellis, C., and Otto, F.}
\newblock Structure of entropy solutions to the eikonal equation.
\newblock {\em J. Eur. Math. Soc. (JEMS) 5}, 2 (2003), 107--145.

\bibitem{dkmo01}
{\sc DeSimone, A., M{\"u}ller, S., Kohn, R.~V., and Otto, F.}
\newblock A compactness result in the gradient theory of phase transitions.
\newblock {\em Proc. Roy. Soc. Edinburgh Sect. A 131}, 4 (2001), 833--844.

\bibitem{GL}
{\sc Ghiraldin, F., and Lamy, X.}
\newblock Optimal {Besov} differentiability for entropy solutions of the
  eikonal equation.
\newblock {\em Commun. Pure Appl. Math. 73}, 2 (2020), 317--349.

\bibitem{jabinperthame01}
{\sc Jabin, P.-E., and Perthame, B.}
\newblock Compactness in {G}inzburg-{L}andau energy by kinetic averaging.
\newblock {\em Comm. Pure Appl. Math. 54}, 9 (2001), 1096--1109.

\bibitem{JK00}
{\sc Jin, W., and Kohn, R.~V.}
\newblock Singular perturbation and the energy of folds.
\newblock {\em J. Nonlinear Sci. 10}, 3 (2000), 355--390.

\bibitem{stab}
{\sc Lamy, X., and Marconi, E.}
\newblock {Stability of the vortex in micromagnetics and related models}.
\newblock {\em arXiv:2209.09662\/} (2022).

\bibitem{LO18}
{\sc Lamy, X., and Otto, F.}
\newblock {On the regularity of weak solutions to Burgers' equation with finite
  entropy production }.
\newblock {\em Calc. Var. Partial Differential Equations 57}, 4 (2018).

\bibitem{marconi21ellipse}
{\sc Marconi, E.}
\newblock Characterization of minimizers of {A}viles-{G}iga functionals in
  special domains.
\newblock {\em Arch. Ration. Mech. Anal. 242}, 2 (2021), 1289--1316.

\bibitem{marconi21micromag}
{\sc Marconi, E.}
\newblock Rectifiability of entropy defect measures in a micromagnetics model.
\newblock {\em Advances in Calculus of Variations\/} (2021),
  000010151520210012.

\bibitem{marconi22structure}
{\sc Marconi, E.}
\newblock On the structure of weak solutions to scalar conservation laws with
  finite entropy production.
\newblock {\em Calc. Var. Partial Differ. Equ. 61}, 1 (2022), 30.
\newblock Id/No 32.

\bibitem{RS01}
{\sc Rivi\`ere, T., and Serfaty, S.}
\newblock Limiting domain wall energy for a problem related to micromagnetics.
\newblock {\em Comm. Pure Appl. Math. 54}, 3 (2001), 294--338.

\end{thebibliography}

\end{document}